\documentclass[12pt]{amsart}

\makeatletter

\def\chaptermark#1{}

\def\chapter{%
  \if@openright\cleardoublepage\else\clearpage\fi
  \thispagestyle{plain}\global\@topnum\z@
  \@afterindenttrue \secdef\@chapter\@schapter}

\def\@chapter[#1]#2{\refstepcounter{chapter}%
  \ifnum\c@secnumdepth<\z@ \let\@secnumber\@empty
  \else \let\@secnumber\thechapter \fi
  \typeout{\chaptername\space\@secnumber}%
  \def\@toclevel{0}%
  \ifx\chaptername\appendixname \@tocwriteb\tocappendix{chapter}{#2}%
  \else \@tocwriteb\tocchapter{chapter}{#2}\fi
  \chaptermark{#1}%
  \addtocontents{lof}{\protect\addvspace{10\p@}}%
  \addtocontents{lot}{\protect\addvspace{10\p@}}%
  \@makechapterhead{#2}\@afterheading}
\def\@schapter#1{\typeout{#1}%
  \let\@secnumber\@empty
  \def\@toclevel{0}%
  \ifx\chaptername\appendixname \@tocwriteb\tocappendix{chapter}{#1}%
  \else \@tocwriteb\tocchapter{chapter}{#1}\fi
  \chaptermark{#1}%
  \addtocontents{lof}{\protect\addvspace{10\p@}}%
  \addtocontents{lot}{\protect\addvspace{10\p@}}%
  \@makeschapterhead{#1}\@afterheading}
\newcommand\chaptername{Chapter}

\def\@makechapterhead#1{\global\topskip 7.5pc\relax
  \begingroup
  \fontsize{\@xivpt}{18}\bfseries\centering
    \ifnum\c@secnumdepth>\m@ne
      \leavevmode \hskip-\leftskip
      \rlap{\vbox to\z@{\vss
          \centerline{\normalsize\mdseries
              \uppercase\@xp{\chaptername}\enspace\thechapter}
          \vskip 3pc}}\hskip\leftskip\fi
     #1\par \endgroup
  \skip@34\p@ \advance\skip@-\normalbaselineskip
  \vskip\skip@ }
\def\@makeschapterhead#1{\global\topskip 7.5pc\relax
  \begingroup
  \fontsize{\@xivpt}{18}\bfseries\centering
  #1\par \endgroup
  \skip@34\p@ \advance\skip@-\normalbaselineskip
  \vskip\skip@ }
\def\appendix{\par
  \c@chapter\z@ \c@section\z@
  \let\chaptername\appendixname
  \def\thechapter{\@Alph\c@chapter}}

\newcounter{chapter}

\newif\if@openright

\makeatother

\usepackage{amsmath}
\usepackage{tikz-cd}
\usepackage{mathtools}
\usepackage{amsfonts}
\usepackage{graphicx}
\usepackage{verbatim}
\usepackage{textcomp}\usepackage{amssymb}
\usepackage{cite}
\usepackage{color}
\usepackage[all]{xy}
\usepackage{graphicx}
\usepackage{color}
\usepackage{hyperref}
\usepackage{soul}
\usepackage[normalem]{ulem}
\hypersetup{
    colorlinks,
    citecolor=black,
    filecolor=black,
    linkcolor=black,
    urlcolor=black
}
\usepackage{tikz}
\usepackage[all]{xy}
\usepackage{graphicx}
\usetikzlibrary{backgrounds}
\newtheorem{theorem}{Theorem}[section]
\newtheorem{lemma}[theorem]{Lemma}

\newtheorem{corollary}[theorem]{Corollary}

\allowdisplaybreaks
\newtheorem{notation}[theorem]{Notation}

\usepackage{amsthm}

\DeclareMathOperator*{\essinf}{ess\,inf}

\theoremstyle{definition}
\newtheorem{definition}[theorem]{Definition}

 \newenvironment{proof1.2}
{{\it Proof of  Theorem~\ref{uniqueness}.}}
{\hfill $\Box$ \\}

\numberwithin{equation}{section}

\newcommand{\R}{\mathbb R}
\newcommand{\Z}{\mathbb Z}
\newcommand{\RR}{\mathcal R}

\newcommand{\C}{\mathbb C}
\newcommand{\D}{\mathbb D}
\newcommand{\Sp}{\mathbb S}

\newcommand{\B}{\mathcal B}

\newcommand{\HH}{\mathcal H}

\newcommand{\LL}{\mathcal L}

\newcommand{\domain}{M}

\newcommand{\DDD}{\widetilde{\bar \D^*}}
\newcommand{\DDDD}
{\bar \D^* \rightarrow \DDD \times_{\rho} \tilde X}

\newcommand{\Ej}{\mathsf E_j}

\title{Infinite energy harmonic maps from Riemann surfaces to CAT(0) spaces}

\author[Daskalopoulos]{Georgios Daskalopoulos}
\address{Department of Mathematics \\
                 Brown Univeristy \\
                 Providence, RI}%02912}
\email{daskal@math.brown.edu}
\author[Mese]{Chikako Mese}
\address{Johns Hopkins University\\
Department of Mathematics\\
Baltimore, MD}%  21218}
\email{cmese@math.jhu.edu}

\begin{document}

\thanks{
GD supported in part by NSF DMS-2105226, CM supported in part by NSF DMS-2005406.}
\maketitle

\begin{abstract}
We present some results about harmonic maps with possibly infinite energy from punctured Riemann surfaces to CAT(0) spaces.  In particular, we give precise estimates of their  energy growth   near the punctures and prove  uniqueness.  
\end{abstract}

\vspace*{0.2in}


\vspace*{0.2in}
\begin{center}
{\sc Table of Contents}
\end{center}

\begin{itemize}

\item[] \S\ref{introduction} Introduction 
	\hfill p.~\pageref{introduction}
\item[] \S\ref{preliminaries} Preliminaries 
	\hfill p.~\pageref{preliminaries}
\item[] \S\ref{disc} Harmonic maps from a punctured disk 
	\hfill p.~\pageref{disc}
\item[] \S\ref{sec:proofexistence} Existence of infinite energy harmonic maps 	
	\hfill p.~\pageref{sec:proofexistence}
\item[] \S\ref{sec:proofuniqueness} Uniqueness of infinite energy harmonic maps	\hfill p.~\pageref{sec:proofuniqueness}
\end{itemize}

\section{Introduction} \label{introduction}

 This is the first in a series of papers where we develop the techniques to study non-abelian Hodge theory on quasi-projective varieties. The sequel is
 \cite{daskal-meseQP} and \cite{ya}.
 More specifically, in this first paper, we  study harmonic maps with possibly infinite energy from punctured Riemann surfaces into  complete CAT(0) spaces and we a give precise estimate of their  energy growth  near the punctures. We also prove  a uniqueness result.  Both results are crucial for the next papers.

Throughout this paper, we use the following notation unless explicitly stated otherwise:
\begin{itemize}
\item $\tilde X$ is a complete CAT(0)-space
\item $\bar \RR$ is a compact Riemann surface 
\item $\RR=\bar \RR \backslash \{p^1, \dots, p^n\}$ is a punctured Riemann surface
\item $\Pi:\tilde \RR \rightarrow \RR$ is the universal covering map
\item $\rho:\pi_1(\RR) \rightarrow \mathsf{Isom}(\tilde X)$ is a homomorphism.
\end{itemize}
For each  puncture $p^j$, $j=1, \dots, n$, 
\begin{itemize}
\item $\D^j$ is a conformal disks in $\RR$ centered at the  puncture $p^j$ ($\D^j \cap \D^i =\emptyset$ for $i\neq j$)
\item $\D^j_r=\{z \in \D:  |z|<r\}$ 

\item $\lambda^j \in \pi_1(\RR)$ is the conjugacy class associated to the loop in $\RR$ around $p^j$,
\item $I^j:=\rho(\lambda^j) \in \mathsf{Isom}(\tilde X)$, and 
\item 
$\Delta_{I^j}: =\inf_{P \in \tilde X} d(I^j(P), P)$ is the translation length of $I^j$.
\end{itemize}
Recall that  $I^j$ is said to be  {\it semisimple}  if there exists $P_* \in \tilde X$ such that $d(I^j(P_*),P_*)=\Delta_{I^j}$.

We prove the following two theorems.

\begin{theorem}[Existence] \label{existence}
Assume
\begin{itemize}
 \item[(A)] The action of $\rho(\pi_1(\RR))$ does not fix a point at infinity.
    \item[(B)] The isometry $I^j$  is semisimple or  there exist a  geodesic ray $c:[0,\infty) \rightarrow \tilde X$ and constants $a,b>0$ such that
\[
d^2(I^j(c(t)), c(t)) \leq  \Delta_{I^j}^2(1+be^{-at}).
\]
 \end{itemize}
 Then there exists a $\rho$-equivariant  harmonic map $u:   \tilde \RR \rightarrow \tilde X$  with logarithmic energy growth towards the punctures, i.e.
 \begin{equation} \label{Cbddef}
\sum_{j=1}^n
\frac{\Delta^2_{I^j}}{2\pi} \log \frac{1}{r} \leq E^u[\RR_r] \leq \sum_{j=1}^n \frac{\Delta^2_{I^j}}{2\pi} \log \frac{1}{r}
\end{equation}
where $
\RR_r  = \RR \backslash \bigcup_{j=1}^n \D^j_r
$
and $E^u[\RR_r]$ is the energy of $u$ in $\RR_r$.
\end{theorem}

For smooth targets, the existence  essentially follows from \cite{lohkamp}, \cite{wolf} and \cite{koziarz-maubon}.  The new statement in Theorem~\ref{existence} is the precise growth estimate towards the punctures of the harmonic map.  

We also prove the following uniqueness result.  

\begin{theorem}[Uniqueness] \label{uniqueness}
Assume (A) and (B) of Theorem~\ref{existence}.  If the $\rho$-equivariant harmonic map $\tilde u: \tilde \RR \rightarrow \tilde X$ has logarithmic energy growth (cf.~(\ref{Cbddef})), then it is the  unique harmonic map with logarithmic energy growth in the following cases: 
\begin{itemize}
\item[(i)] $\tilde X$ is a negatively curved space (i.e.~for any $P \in \tilde X$, there exists a neighborhood $U$ of $P$ which is $CAT(-\kappa)$ for some $\kappa>0$),\item[(ii)] $\tilde X$ is an irreducible symmetric space of non-compact type, or
\item[(iii)] 
$\tilde X$ is an irreducible locally finite Euclidean building such that the action of $\rho(\pi_1(\RR))$ does not fix an unbounded closed convex strict subset  of $\tilde X$.
\end{itemize}
 \end{theorem}

 In the case when the harmonic map has finite energy,  uniqueness  follows from previous works of the authors and Corlette (\cite{mese}, \cite{corlette}, \cite{daskal-meseUnique} for cases (i), (ii), (iii) respectively).
In  \cite{daskal-meseQP},   we use the results of this paper to develop  non-Abelian Hodge theory techniques extending the works by Gromov-Schoen \cite{gromov-schoen} from the finite to the infinite energy case and Mochizuki \cite{mochizuki-memoirs} from the Archimedean case to the non-Archimedean case.  The upcoming joint paper \cite{ya} contains an application of this technique to non-Abelian Hodge theory.\\

{\it Summary of the paper}.  For the purpose of explaining the main ideas,  assume that $\tilde X$ is a universal cover of a compact manifold $X$ and let  $\gamma:\Sp^1 \rightarrow X$ be  an arc length parameterization of a minimizing closed geodesic in $X$.  Denote the energy of $\gamma$ by $E^\gamma$.   A prototype map $v$  corresponding to $\gamma$ is 
\[
v: [0, \infty)  \times \Sp^1 \rightarrow X
\]
which is  equal to $\gamma$ on slices $\{t\} \times \Sp^1$ for  all sufficiently large $t$.   (We  note that the idea of the prototype map first appears in \cite{lohkamp}.) 
The energy of $v$ on the finite cylinder $[0,T] \times \Sp^1$ is bounded from above by $T \cdot E^\gamma +C$ where $C$ is a constant independent of $T$.  Let 
\[
u_T:[0,T] \times \Sp^1 \rightarrow X
\]
be the Dirichlet solution on the finite cylinder $[0,T] \times \Sp^1$  with boundary values equal to $v$.  Since $u_T$  is an energy minimizing map, the energy of $u_T$ is bounded from above by $T \cdot E^\gamma +C$.   Since $\gamma$ is also an energy minimizing map, the energy of $u_T$ in any finite cylinder $[0,t] \times \Sp^1$ for $0< t \leq T$ is bounded from below by $t \cdot E^\gamma$.  From this, we conclude that the energy of $u_T$ on any finite cylinder $[t_1,t_2] \times \Sp^1$ has an upper  bound independent of $T$; namely,  $(t_2-t_1) \cdot E^\gamma + C$.
 Thus,  the regularity theory for harmonic maps implies  that the family of maps $\{u_T\}_{T\geq 1}$ has a uniform Lipschitz bound on any compact subset of $[0,\infty) \times \Sp^1$.  This implies  that there exists a sequence $T_i \rightarrow \infty$ such that  $u_{T_i}$ converges locally uniformly to a harmonic map $u:[0,\infty) \times \Sp^1 \rightarrow  X$. 
 The conformal equivalence of the punctured disk $\D^*$ and the infinite cylinder $[0,\infty) \times \Sp^1$ thus defines a harmonic map on $\D^*$ 
 with lower and upper energy bound 
 \begin{equation} \label{lu}
 E^\gamma \log \frac{1}{r} \leq E^u[\D_r] \leq E^\gamma \log \frac{1}{r} +C.
 \end{equation}
  We extend this  idea  to prove the existence of  equivariant  harmonic maps from any punctured Riemann surface $\Sigma$.  
 
 For  finite energy harmonic maps, the uniqueness can be derived from the fact that the energy functional is a convex function along a geodesic interpolation.  Namely, if $u_t$ is the geodesic interpolation map (i.e.~$t \mapsto u_t(x)$ is a geodesic), then $E^{u_t} \leq (1-t) E^{u_0} + t E^{u_1}$. If $u_0$ and $u_1$ are energy minimizing maps, then the convexity implies that $t \mapsto E^{u_t}$ is a constant function.  This idea does not directly generalize to our situation because we deal with infinite energy maps. The main idea to prove the uniqueness assertion in our setting is to introduce a modified energy functional by subtracting off the logarithmic energy growth near the punctures.  We then prove  that the modified energy  is constant  along the geodesic interpolation of harmonic maps. This allows us to apply the  argument used by the authors in \cite{mese},  \cite{daskal-meseUnique}.

\section{Preliminaries} \label{preliminaries}

\subsection{Maps into CAT(0) spaces}
\label{subsec:korevaarschoen}

\begin{definition} \label{def:NPC}
A say that $\tilde X$ is a {\it complete CAT(0) space} if it is a complete geodesic space that satisfies the 
 following condition:
  For any three points $P,R,Q \in \tilde X$ and an arclength parameterized geodesic $c:[0,l]
\rightarrow \tilde X$ with $c(0)=Q$ and $c(l)=R$, 
\[
d^2(P,Q_t) \leq (1-t) d^2(P,Q)+td^2(P,R)-t(1-t)d^2(Q,R)
\]
where $Q_{t}=c(tl)$.  (We refer to \cite{bridson-haefliger} for more details.)
\end{definition}

 \begin{notation} \label{interpolationnotation}
\emph{It follows immediately from Definition~\ref{def:NPC} that, given $P, Q \in \tilde X$ and $t \in [0,1]$, there exists a {\it unique} point with distance from $P$ equal to $td(P,Q)$ and the distance from $Q$ equal to $(1-t)d(P,Q)$.  We denote this point by
\[
(1-t)P+tQ.
\]
}
\end{notation}

Fix a smooth conformal Riemannian metric $g$ on $\RR$ (i.e.~$g$ is given by $\rho(z)\, dz d\bar z$ for some smooth function $\rho$ in any holomorphic coordinate $z$ of $\RR$).  We refer to \cite{korevaar-schoen1} for the notion of energy  $E^f$ and energy density function $|\nabla f|^2$ of a map $f:\RR \rightarrow \tilde X$. 
The (Korevaar-Schoen) energy of $f$ in a domain $\Omega \subset \RR$ is 

\[
E^f[\Omega] =
\int_{\Omega}|\nabla f|^2 d\mbox{vol}_g.
\]
 It is well-known that $E^f[\Omega]$ and $|\nabla f|^2 d\mbox{vol}_g$ is invariant of the choice of the smooth conformal Riemannian metric $g$.  

We say a continuous map $u: \RR \rightarrow \tilde X$  is {\it harmonic} if it is locally energy minimizing;   more precisely, at each $p \in \RR$, there exists a neighborhood $\Omega$ of $p$  so that all continuous comparison maps which agree with $u$ outside of this neighborhood have no less energy.

For  $V \in \Gamma \Omega$ where $\Gamma \Omega$ is the set of
Lipschitz vector fields  on $\Omega$, $|f_*(V)|^2$  is similarly
defined.  The real valued $L^1$ function $|f_*(V)|^2$ generalizes
the norm squared on the directional derivative of $f$.  The
generalization of the pull-back metric is the continuous, symmetric, bilinear, non-negative and tensorial operator
\[
\pi_f(V,W)=\Gamma \Omega \times \Gamma \Omega \rightarrow
L^1(\Omega, {\bf R})
\]
where
\[
\pi_f(V,W)=\frac{1}{2}|f_*(V+W)|^2-\frac{1}{2}|f_*(V-W)|^2.
\]
  We refer to \cite{korevaar-schoen1} for more
details.

Fix a local chart   $(U; z)$.  By the conformal invariance of the energy,    we can assume that the metric $g$ the Eucliean metric  $dzd\bar z$ in $U$.  We extend $\pi_f$ linearly cover $\C$.
Then energy density function of
$f$ is given by
\[
\frac{1}{4}|\nabla f|^2=  \pi_f(\frac{\partial f}{\partial z^i},
\frac{\partial f}{\partial \bar z^j}).
\]
Furthermore, set $z=re^{i\theta}$ to define polar coordinates $(r,\theta)$.  We define 
\[
 \left| \frac{\partial u}{\partial r} \right|^2:=  \pi_f(\frac{\partial }{\partial r},
\frac{\partial }{\partial r}) \ \mbox{ and } \ \left| \frac{\partial u}{\partial \theta} \right|^2:=  \pi_f(\frac{\partial }{\partial \theta},
\frac{\partial }{\partial \theta}).
\]

\subsection{Equivariant maps and section of the flat $\tilde X$-bundle} \label{subsec:donaldsonsection}
\begin{definition} \label{def:equivariant} Let $\rho: \pi_1(M) \rightarrow \mathsf{Isom}(\tilde X)$. A map $\tilde f:\tilde \domain \rightarrow  \tilde X$ is said to be $\rho$-equivariant if 
\[
\tilde f(\gamma p) = \rho(\gamma) \tilde f(p), \ \ \forall \gamma \in \pi_1(M), \ p \in \tilde \domain.
\]
\end{definition}
  
 Following Donaldson  \cite{donaldson},  we will replace equivariant   maps with  sections of an associated fiber bundle. 
The quotient under the action of $\pi_1(M)$ on the product $\tilde M \times \tilde X$  is  the {\it twisted product}
\[
\tilde M \times_\rho \tilde X.
\]
In other words, $\tilde M \times_\rho \tilde X$ is the set of orbits  $[(p, x)]$ of a point $(p,x) \in  \tilde M \times  \tilde X$  under the action of $\gamma \in \pi_1(M)$ via   the deck transformation  on the first component and the isometry $\rho(\gamma)$ on the second component.  The fiber bundle
\[
\tilde M \times_\rho \tilde X \rightarrow M
\]
with  fibers over $p \in M$ is isometric to $\tilde X$ is the {\it flat $\tilde X$-bundle} over $M$ defined by $\rho$.

There is a one-to-one correspondence between sections of this fibration   and $\rho$-equivariant maps 
\[
 \tilde f:  \tilde M \rightarrow \tilde X 
 \ \ \ \ \longleftrightarrow \ \ \ \ 
f: M  \rightarrow \tilde M \times_{\rho} \tilde X
 \]
satisfying the relationship  
\[
[(\tilde p, \tilde f(\tilde p))] \leftrightarrow f(p)
\ \mbox{
 where $\Pi(\tilde p) =p$.}
 \]
Since the energy density function  $|\nabla \tilde f|^2$ of $\tilde f$ is a $\rho$-invariant function, we can define
\[
|\nabla f|^2(p):=|\nabla \tilde f|^2(\tilde p). 
\]

We can similarly define the pullback inner product and directional energy density functions of $f$  by using the corresponding notions for $\tilde f$ given in Subsection~\ref{subsec:korevaarschoen}.
For $U \subset M$, 
the (vertical) energy of a section $f$ is 
\begin{equation} \label{vertical}
E^f[U] = \int_U |\nabla f|^2 d\mbox{vol}_g.
\end{equation}
Furthermore,
for   sections $f_1$, $f_2$, we define
\begin{equation} \label{defdis}
d(f_1(p),f_2(p)):=d(\tilde f_1(\tilde p), \tilde f_2(\tilde p))
\end{equation}
where  $\tilde f_1$,  $\tilde f_2$ are  the associated  $\rho$-equivariant maps  to sections $f_1$, $f_2$ respectively.

\subsection{Sublogarithmic growth}

We will
\begin{equation}  \label{fixeddisk}
\mbox{fix a conformal disk $\D^j \subset \bar \RR$ centered at each puncture $p^j$}
\end{equation}
such that $\D^i \cap \D^j=\emptyset$ for $i\neq j$.  Furthermore, let $$\D^{j*}=\D^j \backslash \{0\}$$.

Fix $P_0 \in \tilde X$ and a fundamental domain $F$ of $\tilde \RR$.  Let  $f_0$ be the section of the fiber bundle $\tilde \RR \times_\rho \tilde X \rightarrow \RR$  such that, for any $p \in \RR \cap \Pi(F)$,   $f_0(p)=[(\tilde p,P_0)]$ where $\tilde p = \Pi^{-1}(p) \cap F$.  (Note that  $\Pi(F)$ is of full measure in $\RR$.) 

For a given section $f:\RR \rightarrow \RR \times_\rho \tilde X$,  define
\begin{equation} \label{tildeyes}
\delta_j:\D^{j*} \rightarrow [0,\infty), \ \ \ \ \delta_j(z) = \essinf_{\{  z \in \D^{j*}\}} d( f(z),f_0(z)).
\end{equation}
Recall that $d( f(z),f_0(z))$ is defined by (\ref{defdis}).
\begin{definition} \label{logdec} 
We say  a section $f:\RR \rightarrow \tilde \RR \times_\rho \tilde X$  (or its associated equivariant map) has \emph{sub-logarithmic growth} if for any $j=1, \dots, n$
\[
\lim_{|z| \rightarrow 0} \delta_{j}(z) +\epsilon \log |z|=-\infty \mbox{ in }\D^{j*} .
\]
By the triangle inequality, this definition is independent of the choice of  $P_0 \in \tilde X$. We say that $f$ has {\it logarithmic energy growth} if near the punctures satisfies condition (\ref{Cbddef}) in Theorem~\ref{existence}. 
 \end{definition}

\section{Harmonic maps from a punctured disk}\label{disc}

Before we consider the domain to be a general punctured Riemann surface, we first start by studying harmonic maps from a punctured disk $\bar \D^*$.  Denote 
\begin{eqnarray*}
\D_r^*& = & \{z \in \D: |z|<r\}
\\
\D_{r,r_0} & = & \D_{r_0} \backslash \D_r.
\end{eqnarray*}
Identify the boundary of the disk with the circle; i.e.
\[
\partial \D=\Sp^1,
\]
This induces a natural identification  
\[
2\pi \Z \simeq \pi_1(\Sp^1) \simeq  \pi_1( \bar \D^*). 
\]
The main result of this section is the following.

\begin{theorem}[Existence and Uniqueness of the Dirichlet solution on $\D^*$] \label{exists}
Assume the following:
\begin{itemize}
\item $\rho:2\pi \Z \simeq \pi_1(\Sp^1) \simeq  \pi_1( \bar \D^*) \rightarrow \mathsf{Isom}(\tilde X)$  
is a homomorphism

 \item $k:\DDDD$ is a locally Lipschitz section
 \item $I:=\rho([\Sp^1])$ satisfy condition (B) of Theorem~\ref{existence} with $I^j$ replaced by $I$.
  \end{itemize}
Then there  exists a  harmonic section  
\[
u:\DDDD \mbox{  with  }u|_{\Sp^1} =k|_{\Sp^1}.
\]
  Furthermore,  
there exists a constant $C>0$   that depends only on  $\mathsf{E}_\rho=\frac{\Delta_I^2}{2\pi}$,  $a$, $b$ from (B) if $I$ is not semisimple and the section $k$ satisfying the following properties:\\

\begin{itemize}
\item[(i)]
$\displaystyle{\mathsf{E}_\rho \log \frac{1}{r} \leq E^u[\D_{r,1}] \leq \mathsf{E}_\rho \log \frac{1}{r}+C, \ \ \ 0<r\leq 1}$
\item[(ii)]  $\displaystyle{\left| \frac{\partial u}{\partial r} \right|^2 \leq \frac{C}{r^2(-\log r)} \ \  \mbox{ and } \ \ 
\left( \left| \frac{\partial u}{\partial \theta} \right|^2 - \frac{\mathsf{E}_\rho}{2\pi} \right)   \leq \frac{C}{-\log r}
}$ \ in \  $\D_{\frac{1}{4}}^*$
 \item[(iii)]  $u$ has sub-logarithmic growth. 
\end{itemize}
Moreover, $u$ is the only harmonic section satisfying $u|_{\partial \D} =k|_{\partial \D}$ and property (iii).
\end{theorem}

\begin{proof}
Combine Lemma~\ref{(i)}, Lemma~\ref{(ii)}, Lemma~\ref{(iii)} and Lemma~\ref{uniquenessdisk} below.
\end{proof}
The purpose for the rest of this subsection is to prove the lemmas that comprise the proof of Theorem~\ref{exists}.   
We start with the following preliminary lemma about subharmonic functions.

\begin{lemma} \label{shdisk}
For $r>0$, let $\nu: \D_r^* \rightarrow \R$ be a  subharmonic function that extends as a continuous function to $\bar \D_r^*= \overline{\D_r} \backslash \{0\}$. If we assume that 
\begin{equation} \label{slowblowup}
\lim_{z \rightarrow 0} \  \nu(z) + \epsilon \log |z| =-\infty, \ \ \forall \epsilon>0,
\end{equation}
 then 
 \[
\sup_{z \in \D_r^*} \nu(z)  \leq \max_{\zeta \in \partial \D_r} \nu(\zeta).
\] 
In particular, $\nu$ extends to a subharmonic function on $\D_r$ (cf.~\cite{hayman-kennedy}).
 \end{lemma}

\begin{proof}
Let  $h:\D_r \rightarrow \R$ be  the unique Dirichlet solution  with $h|_{\partial \D_r}=\nu|_{\partial \D_r}$.  By the maximum principle,  $h$ is non-negative and bounded on $\D_r$.  The function 
\[
f_{\epsilon}(z)=\nu(z)-h(z)+\epsilon (\log |z|-\log r)
\]
is subharmonic on $\D_r^*$ with 
$
f_{\epsilon}|_{\partial \D_r} \equiv 0$ on $\partial \D_r$.
Furthermore, since $h$ is a bounded function and $\nu$ satisfies (\ref{slowblowup}), 
\[ 
\lim_{z \rightarrow 0} f_{\epsilon}(z) =-\infty.
\]
By the maximum principle, $f_{\epsilon} \leq 0$.  By letting $\epsilon \rightarrow 0$, we conclude $\nu(z) -h(z) \leq 0$  for all $z \in \D_r^*$.
Thus, 
\[
\sup_{z \in \D_r^*} \nu(z)  \leq \sup_{z \in \D_r}h(z) \leq    \max_{\zeta \in \partial \D_r} h(\zeta)=\max_{\zeta \in \partial \D_r} \nu(\zeta).
\] 
\end{proof}

\begin{lemma} \label{lowerbd}
If  $\rho$ and $I$ are as in Theorem~\ref{exists}, then for any section $f: \D_{r,r_0} \rightarrow \DDD \times_{\rho} \tilde X$,
\[
\mathsf{E}_\rho \log \frac{r_0}{r}  \leq 
  \int_{\D_{r,r_0}}  \frac{1}{r^2} \left| \frac{\partial f}{\partial \theta} \right|^2 rdr \wedge d\theta \leq E^{f}[\D_{r,r_0}], \ \ 0<r<r_0<1. 
\] 
\end{lemma}

\begin{proof}
By the definition of $\mathsf{E}_I$, any section  $c:\Sp^1 \rightarrow \R \times_{\rho} \tilde X$ satisfies
\[
\mathsf{E}_\rho
 \leq   \int_0^{2\pi} \left| \frac{\partial c}{\partial \theta} \right|^2 d\theta.
 \]
After identifying $\Sp^1$ with $\partial \D_r$,  we can view the restriction  $f|_{\partial \D_r}$ as a section  $\Sp^1 \rightarrow \R \times_{\rho} \tilde X$.  
Thus, 
\begin{eqnarray*}\label{ggeo} 
\mathsf{E}_\rho \log \frac{r_0}{r} & = & \mathsf{E}_\rho  \int_r^{r_0} \frac{1}{r^2}  rdr \wedge d\theta \nonumber\\
&  \leq &   \int_0^{2\pi} \int_r^{r_0} \frac{1}{r^2} \left| \frac{\partial f}{\partial \theta} (r,\theta) \right|^2 rdr \wedge d\theta \\
& \leq & E^f[\D_{r,r_0}] \nonumber.
\end{eqnarray*}
\end{proof}

\begin{lemma} \label{defineg}
If $\rho$, $I$ and $k:\DDDD$ are as in  Theorem~\ref{exists}, then there exists a constant $C>0$ and a  Lipschitz section $v:\DDDD$  with $v|_{\partial \D}=k|_{\partial \D}$ such that 
\begin{equation}   \label{glog}
\mathsf{E}_\rho \log \frac{r_0}{r} \leq  E^v[\D_{r,r_0}]  \leq \mathsf{E}_\rho \log \frac{r_0}{r}+C, \ \ \ 0<r<r_0\leq \frac{1}{2}.
\end{equation}
Moreover, the Lipschitz constant of $v$ and $C$ are dependent only on $\mathsf{E}_\rho=\frac{\Delta_I^2}{2\pi}$ and  $a$, $b$ form assumption (B) if $I$ is not semisimple and $k$.
\end{lemma}

\begin{proof} 
We will  first construct a one-parameter family of  sections   $\gamma_s:\Sp^1 \mapsto \R \times_{\rho} \tilde X$ as follows:
\begin{itemize}
\item If $I$ is semisimple, let $P_{\star} \in \tilde X$ be a point where the infimum $\Delta_I$ is attained.  For any $s \in [0,\infty)$, define 
$\tilde \gamma_s:\R \rightarrow \tilde X$ to be the $\rho$-equivariant geodesic  (a constant map if $I$ is elliptic) such that $\tilde \gamma_s(0)=P_{\star}$ and  $\tilde \gamma_s(2\pi)=I(P_{\star})$ and  let $\gamma_s:\Sp^1 \rightarrow \tilde \Sp^1 \times_{\rho} \tilde X$ be the associated section.
\item
Otherwise, 
let  
 $c:[0,\infty) \rightarrow \tilde X$ be a geodesic ray defined in assumption (B).  Define the $\rho$-equivariant curve 
$\tilde \gamma_s:\R  \rightarrow \tilde X$ such that $\tilde \gamma_s|_{[0,2\pi]}$ is  the geodesic from $c(s)$ and $I(c(s))$ and  let $\gamma_s:\Sp^1 \rightarrow \tilde \Sp^1 \times_{\rho} \tilde X$ be the associated section.
\end{itemize}
We note that by the quadrilateral comparison of CAT(0) spaces (cf.~\cite{reshetnyak}), for $\theta
\in [0,2\pi]$,
\[
d(\tilde \gamma_{s_1}(\theta), \tilde \gamma_{s_2}(\theta)) \leq (1-\frac{\theta}{2\pi}) d(c(s_1), c(s_2))+\frac{\theta}{2\pi}d(I \circ c(s_1)), I \circ c(s_2)))=|s_1-s_2|.
\]
 Thus
 \begin{equation} \label{npcquad}
\left| \frac{d}{ds} (\tilde \gamma_s(\theta)) \right| \leq 1, \ \ \ \forall  s \in (0,\infty), \  \theta \in \Sp^1.
\end{equation} 
By the assumption on $c(s)$,  we have
\begin{equation} \label{gammatbd}
\int_{\Sp^1} \left| \frac{\partial \gamma_s}{\partial \theta} \right|^2 d\tau \leq \mathsf{E}_\rho(1+2\pi be^{-as}).
\end{equation} 

Define a $\rho$-equivariant map $\tilde v:  \DDD \rightarrow \tilde X$ as follows:  First, for $(r,t) \in  \DDD$,  let
\[
\tilde v(r,t):=\tilde \gamma_{(-\log r-\log 2)^{\frac{1}{3}}}(t) \mbox{ for } r \in (0,\frac{1}{2}] 
\  \mbox{ and } \ 
 \tilde v(1,t):=\tilde k(1,t).
\] 
Next,  for each $t \in \R$, let 
\[
r \mapsto  \tilde v(r,t) \mbox{ for } r \in [\frac{1}{2},1]
\]
to be the arclength parameterization of the  geodesic between $\tilde v(\frac{1}{2}, t)$ and $\tilde k(1,t)$.  Finally, let
$v:\DDDD$ be the section associated to the $\rho$-equivariant map $\tilde v$.

For $0<r<r_0<\frac{1}{2}$, we have
\begin{eqnarray*}
\mathsf{E}_\rho \log \frac{r_0}{r} & \leq  &  \int_0^{2\pi} \int_r^{r_0} \frac{1}{r^2} \left| \frac{\partial v}{\partial \theta} \right|^2 rdr \wedge d\theta \ \ \ \ \  \mbox{(by Lemma~\ref{lowerbd})} \nonumber \\
& \leq  &  \int_0^{2\pi} \int_r^{r_0} \frac{1}{r^2}  \left| \frac{\partial (\gamma_{(-\log r-\log 2)^{\frac{1}{3}}})}{\partial \theta} \right|^2 rdr \wedge d\theta
\\
& =  &  \int_0^{2\pi} \int_{-\log r_0}^{-\log r}  \left| \frac{\partial (\gamma_{(t-\log 2)^{\frac{1}{3}}})}{\partial \theta} \right|^2 dtd\theta \nonumber \\
& \leq  &  \int_{-\log r_0}^{-\log r}( \mathsf{E}_\rho+be^{-a(t-\log 2)^{\frac{1}{3}}})
 dt \ \ \ \ \ \ \ \ \mbox{(by (\ref{gammatbd}))}\nonumber \\
& = & 
\mathsf{E}_\rho \log \frac{r_0}{r}+C.  
\end{eqnarray*}
By (\ref{npcquad}), 
\begin{eqnarray*}
\left| \frac{\partial v}{\partial r} \right|^2(r,\theta)  =  \left| \frac{\partial}{\partial s} \Big|_{s=(-\log r-\log 2)^{\frac{1}{3}} } \gamma_s(\theta) \right|^2  \left| \frac{\partial ( (-\log r-\log 2)^{\frac{1}{3}} )}{\partial r} \right|^2
\leq  \frac{1}{r^2(-\log r-\log 2)^{\frac{4}{3}}},
\end{eqnarray*}
and thus
\begin{eqnarray*}
  \int_0^{2\pi} \int_r^{r_0} \left| \frac{\partial v}{\partial r} \right|^2 rdr \wedge d\theta  & = &  \int_0^{2\pi} \int_r^{r_0}  \frac{1}{r(-\log r-\log 2)^{\frac{4}{3}}} dr \wedge d\theta  \leq C.  
\end{eqnarray*}
\end{proof}

\begin{lemma}[Existence and property (i) of Theorem~\ref{exists}] \label{(i)}
If $\tilde X$, $\rho$,  $\mathsf{E}_\rho$ and $k$ as in  Theorem~\ref{exists}, 
then there  exists a  harmonic section  $u:\D^* \rightarrow \widetilde{\D^*} \times_\rho \tilde X$  with  $u|_{\partial \D} =k|_{\partial \D}$ and a constant $C>0$ that   depends only on  $\mathsf{E}_\rho=\frac{\Delta_I^2}{2\pi}$,  $a$, $b$ from (B) if $I$ is not semisimple and the section $k$ such that 
\[
\mathsf{E}_\rho  \log \frac{1}{r} \leq E^u[\D_{r,1}] \leq \mathsf{E}_\rho \log \frac{1}{r}+C, \ \ \ 0<r\leq 1.
\]
\end{lemma}

\begin{proof}
Let  $v$ be as in Lemma~\ref{defineg}. By \cite[Theorem 2.7.2]{korevaar-schoen1} (or more specifically, by the proof of this theorem which solves the equivariant problem), there exists a  unique harmonic section 
\[
u_r:   \D_{r,1} \rightarrow  \widetilde{ \D_{r,1}}\times_{\rho} \tilde X
\mbox{ with } u_r|_{\partial \D_{r,1}}=v|_{\partial \D_{r,1}}.
 \] 
We have that
\begin{eqnarray}
\mathsf{E}_\rho \log \frac{r_0}{r} +E^{u_r}[\D_{r_0,1}] & \leq & E^{u_r}[\D_{r,r_0}] +E^{u_r}[\D_{r_0,1}]  \ \ \ \ \  \mbox{(by Lemma~\ref{lowerbd})} \nonumber 
\\
& = & E^{u_r}[\D_{r,1}]\nonumber  \\
 & \leq  &  E^{v}[\D_{r,1}] \ \ (\mbox{since $u_r$ is minimizing}) \nonumber  \\
  & = &  E^{v}[\D_{r,r_0}]+ E^{v}[\D_{r_0,1}]  \nonumber  \\
 & \leq &  
\mathsf{E}_\rho \log \frac{r_0}{r} +C +E^{v}[\D_{r_0,1}] 
 \ \ \ \ \  \mbox{(by (\ref{glog}))}.  \nonumber 
 \end{eqnarray}
Therefore,
 \begin{equation} \label{bdvr}
 E^{u_r}[\D_{r_0,1}] \leq C+ E^{v}[\D_{r_0,1}].
 \end{equation}
Note that the right hand side of the above is independent of $r \in (0,\frac{1}{2}]$.  
Thus, $\{u_r|_{\D_{2r_0,1}}\}$ is an equicontinuous family of sections (cf.~\cite[Theorem 2.4.6]{korevaar-schoen1}). Consequently, there exists a sequence  $r_i \rightarrow 0$ and a harmonic section 
\begin{equation} \label{hmpd}
u:  \DDDD \mbox{ with }  u|_{\partial \D} =k|_{\partial \D}.
\end{equation}
 such that the sequence $\{u_{r_i}\}$ converges uniformly on every  compact subsets of $\D^*$ to  $u$.

To prove property (i),  we note that for $0<r<r_1<1$,  \begin{eqnarray*}
E^{u_r}[\D_{r,r_1}]  +E^{u_r}[\D_{r_1,1}] & = & E^{u_r}[\D_{r,1}] \\
& \leq &   E^{v}[\D_{r,1}]
\\
& = & E^{v}[\D_{r,r_1}]+E^{v}[\D_{r_1,1}] \end{eqnarray*}
which, combined with Lemma~\ref{lowerbd} and (\ref{glog}), imply that
\[
E^{u_r}[\D_{r_1,1}]  \leq E^v[\D_{r_1,1}]+C.
\]
Letting $r \rightarrow 0$ and applying the lower semicontinuity of energy \cite[Theorem 1.6.1]{korevaar-schoen1},  we obtain
\[
E^{u}[\D_{r_1,1}]  \leq E^v[\D_{r_1,1}]+C.
\]
Applying Lemma~\ref{lowerbd} and (\ref{glog}), we obtain
\[
\mathsf{E}_\rho \log \frac{1}{r_1}  \leq E^{u}[\D_{r_1,1}]  \leq E^v[\D_{r_1,1}]+C \leq \mathsf{E}_\rho \log \frac{1}{r_1}+C.
\]
\end{proof}

   \begin{lemma}  \label{integrable1}
  If    $u:\DDDD$ is a harmonic section  with logarithmic energy growth,
then
 \[
\int_{\D^*}  \left| \frac{\partial u}{\partial r}\right|^2 rdr \wedge d\theta  \leq C \ \mbox{ and } \ \int_{\D^*}  \frac{1}{r^2} \left( \left| \frac{\partial u}{\partial \theta}\right|^2 -  \frac{\mathsf{E}_\rho}{2\pi} \right) rdr \wedge d\theta  \leq C
\]
where $C$ that depends $C_0$ from (\ref{Cbddef}). \end{lemma}

 \begin{proof}
 Follows immediately from the inequality (\ref{Cbddef}) in the Definition~\ref{logdec} and Lemma~\ref{lowerbd}.
\end{proof}

   \begin{lemma} [Property (ii) of Theorem~\ref{exists}] \label{(ii)}  If    $u:\DDDD$ is a harmonic section with logarithmic energy growth,
then
\begin{eqnarray} \label{mika1}
\label{mika2}
 \left(\left| \frac{\partial u}{\partial \theta}\right|^2 - \frac{\mathsf{E}_\rho}{2\pi}\right) & \leq &  \frac{C}{(-\log r)}
\ \text{ in } \,
\D_{\frac{1}{4}}^*
\\
\lim_{r \rightarrow 0}  (-\log r)\left(  \left| \frac{\partial u}{\partial \theta} \right|^2(r,\theta) -     \frac{\mathsf{E}_\rho}{2\pi} \right)
 &  = &  0   \label{limitisF}
\\
\left| \frac{\partial u}{\partial r}\right|^2 & \leq &  \frac{C}{r^2 (-\log r)}  
\ \mbox{ in } \,
\D_{\frac{1}{4}}^*
\label{mika1}
\\
 \label{limitisFrad}
\lim_{r \rightarrow 0}  (-\log r) r^2  \left| \frac{\partial u}{\partial r} \right|^2 (r,\theta) &  = &  0.
\end{eqnarray}
where  $C$ that depends $C_0$ from (\ref{Cbddef}). 
\end{lemma}

\begin{proof}
Consider the cylinder
\[
{\mathcal C}
=(0,\infty) \times \Sp^1
\]
and let
\begin{equation} \label{Phi}
\Phi:{\mathcal C} \rightarrow \D^*, \ \ \ \ (t , \psi) = (r=e^{-t}, \theta=\psi). 
\end{equation}

Since $\Phi$ is a conformal map, $u \circ \Phi$ is harmonic.    Thus, the directional energy density functions $\left| \frac{\partial (u \circ \Phi)}{\partial t} \right|^2$ and $\left| \frac{\partial (u \circ \Phi)}{\partial \psi} \right|^2$  are subharmonic by \cite[Remark 2.4.3]{korevaar-schoen1}.
Furthermore,
$\left| \frac{\partial (u \circ \Phi)}{\partial t} \right|^2$ and $\left| \frac{\partial (u \circ \Phi)}{\partial \psi} \right|^2 - \frac{\mathsf{E}_\rho}{2\pi}$ are integrable in ${\mathcal C}=(0,\infty) \times \Sp^1$ by Lemma~\ref{integrable1} and the chain rule.   

 The subharmonicity   of the directional energy density functions implies
\[
0 \leq \int_{(t_1,t_2) \times \Sp^1} \triangle \varphi  \left| \frac{\partial (u \circ \Phi)}{\partial \psi} \right|^2 dtd\psi, \ \ \forall \varphi \in C_c^{\infty}((t_1,t_2) \times \Sp^1), \ \varphi \geq 0.
\]
Letting $\varphi$ approximate the characteristic function of $(t_1,t_2) \times \Sp^1$, we obtain
\[
0 \leq  \frac{d}{dt}\Big|_{t=t_2} \left( \int_{\{t\} \times \Sp^1} \left| \frac{\partial (u \circ \Phi)}{\partial \psi} \right|^2 d\psi \right)-\frac{d}{dt}\Big|_{t=t_1} \left( \int_{\{t\} \times \Sp^1} \left| \frac{\partial (u \circ \Phi)}{\partial \psi} \right|^2 d\psi \right).
\]
In other words, 
\[
F(t) := \int_{\{t\} \times \Sp^1}  \left| \frac{\partial (u \circ \Phi)}{\partial \psi} \right|^2 -   \frac{\mathsf{E}_\rho}{2\pi} \ d\psi, \ \ \ t \in (0,\infty)
\]
 is a convex function.    By Lemma~\ref{integrable1}, $\int_0^{\infty} F(t) dt< \infty$.  
Since $F(t)$ is convex and  integrable,
$F(t)$ is a  decreasing.  Thus,  we obtain  
\begin{equation} \label{bbbb}
tF(t) \leq 2 \int_{\frac{t}{2}}^{t} F(\tau)\ d\tau \leq 2 \int_{\frac{t}{2}}^{\infty} F(\tau)\ d\tau.
\end{equation}
With $\B_1(t_0,\psi_0)$ denoting the unit disk centered at $(t_0,\psi_0) \in \mathcal C$,  the mean value inequality implies
\begin{eqnarray*}
\lefteqn{
t_0 \left( \left| \frac{\partial (u \circ \Phi)}{\partial \psi}\right|^2
-\frac{\mathsf{E}_\rho}{2\pi} 
\right)(t_0,\psi_0)  
}
\\
& \leq &   
\frac{t_0}{\pi} \int_{\B_1(t_0,\psi_0)}\left( \left| \frac{\partial (u \circ \Phi)}{\partial \psi}\right|^2
-\frac{\mathsf{E}_\rho}{2\pi} 
\right) dtd\psi \\
& \leq & 
\frac{t_0}{\pi} \int_{(t_0-1,t_0+1) \times \Sp^1} \left( \left| \frac{\partial (u \circ \Phi)}{\partial \psi}\right|^2
-\frac{\mathsf{E}_\rho}{2\pi} 
\right) dtd\psi 
\\
&  \leq & \frac{1}{\pi}
\frac{t_0}{t_0-1} \int_{(t_0-1,t_0+1) \times \Sp^1} t   \left( \left| \frac{\partial (u \circ \Phi)}{\partial \psi}\right|^2
-\frac{\mathsf{E}_\rho}{2\pi} 
\right)  dt d\psi  
\\
&  = & 
 \frac{1}{\pi}
\frac{t_0}{t_0-1} \int_{t_0-1}^{t_0+1} t  F(t) dt 
\\
&  \leq & 
 \frac{1}{\pi}
\frac{t_0}{t_0-1}   \int_{t_0-1}^{t_0+1} 2 \int_{\frac{t}{2}}^{\infty} F(\tau)\ d\tau dt \ \ \ \mbox{(by  (\ref{bbbb}))}.
\end{eqnarray*}
Since $F(\tau) \geq 0$, 
\[
\int_{\frac{t}{2}}^{\infty} F(\tau)\ d\tau \leq  \int_{\frac{t_0-1}{2}}^{\infty} F(\tau)\ d\tau, \ \ \ t \in (t_0-1,t_0+1).
\]
Combining the above two inequalities and assuming $t_0 \geq 2$ (and thus $\frac{t_0}{t_0-1}\leq 2$),
\begin{eqnarray}
t_0 \left(  \left| \frac{\partial (u \circ \Phi)}{\partial \psi}\right|^2 (t_0,\psi_0)  -\frac{\mathsf{E}_\rho}{2\pi} \right)
& \leq & \frac{4}{\pi} \int_{\frac{t_0-1}{2}}^{\infty} F(\tau)\ d\tau.
\label{inspector}
\end{eqnarray}
Estimates (\ref{mika2})  and (\ref{limitisF}) both follow from (\ref{inspector}) by the chain rule and the fact that  $F(\tau)$ is integrable on $(0,\infty)$ with  integral bound given by Lemma~\ref{integrable1}.

Similar argument for 
\[
G(t):= \int_{\{t\} \times \Sp^1}  \left| \frac{\partial (u \circ \Phi)}{\partial t} \right|^2 d\psi, \ \ \ t \in (0,\infty)
\]
 proves (\ref{mika1}) and (\ref{limitisFrad}).  
 \end{proof}

 \begin{lemma} \label{des}
Let  $u:\DDDD$ be a harmonic section  with logarithmic energy growth.   For any $\epsilon>0$, there exists $\rho_0>0$ such that 
 \[
d^2(u(r_1e^{i\theta}),u(r_0 e^{i\theta})) 
\leq
-\epsilon \log r_1, \ \ \ \forall 0<r_1<r_0\leq \rho_0, \, 0<\theta<2\pi.
\]
\end{lemma}

 \begin{proof}
 Fix $\epsilon>0$.  The  convergence of 
 $ (-\log r) r^2  \left| \frac{\partial u}{\partial r} \right|^2 (r,\theta)$ to 0 as $r \rightarrow 0$ (cf.~(\ref{limitisFrad})) implies that  there exists $\rho_0>0$ such that
 \[
\left| \frac{\partial u}{\partial r} \right| \leq \frac{\sqrt{\epsilon}}{2r(-\log r)^{\frac{1}{2}} }, \ \ r \in (0,\rho_0]. 
\]
Thus, we have for $0<r_1<r_0 \leq \rho_0$, and noting $\rho$-equivariance of $\tilde u$, we have
\begin{eqnarray*}
d^2(u(r_1e^{i\theta}),u(r_0e^{i\theta})) & \leq & 
 \left(  \int_{r_0}^{r_1}
\frac{d}{dr} d(u(r e^{i\theta}),u(r_0 e^{i\theta}))dr \right)^2
\nonumber \\
& \leq  & 
 \left( 
 \int_{r_1}^{r_0} 
\left| \frac{\partial u}{\partial r} \right|(re^{i\theta}) dr 
\right)^2
\nonumber \\
& \leq  & 
\frac{\epsilon}{4}  \left(  \int_{r_1}^{r_0} \frac{dr}{r(-\log r)^{\frac{1}{2}}}   \right)^2 
\nonumber \\
& \leq  & 
\epsilon \left(  (-\log r_1)^{\frac{1}{2}} - (-\log r_0)^{\frac{1}{2}} \right)^2
 \nonumber \\
 & \leq & 
 -\epsilon \log r_1.
\end{eqnarray*}
\end{proof}

\begin{lemma}[Property (iii) of Theorem~\ref{exists}] \label{(iii)}
If    $u:\DDDD$ is a harmonic section with logarithmic energy growth,
 then $u$ has sub-logarithmic growth.
  \end{lemma}
 
 \begin{proof} 
 By Lemma~\ref{des}, for $\epsilon>0$, there exists 
 $r_0>0$ sufficiently small such that
\[
d^2(u(re^{i\theta}), u(r_0e^{i\theta})) \leq  - \frac{\epsilon}{4}\log r, \ \ \ r \in (0,r_0). 
\]
Set $P_0=\tilde u(r_0e^{i\theta_0}) \in \tilde X$ and define $f_0(p)=[(\tilde p, P_0)]$ as in Definition~\ref{logdec}.  For $z =re^{i\theta}$, 
\begin{eqnarray*}
d^2(u(z),f_0(z))+\epsilon \log |z|& = &  
d^2(u(re^{i\theta}),u(r_0e^{i\theta_0}))   +\epsilon \log r \\
& \leq  &  
2d^2(u(re^{\i\theta}),u(r_0e^{i\theta}))+2d^2(u(r_0e^{i\theta}),u(r_0e^{i\theta_0}))   +\epsilon \log r
\\
& \leq &  \frac{\epsilon}{2}\log r +d^2(u(r_0e^{i\theta}),u(r_0e^{i\theta_0})). \end{eqnarray*}
Now let  $r \rightarrow 0$.
 \end{proof}

\begin{lemma}[Uniqueness for sub-logarithmic growth maps] \label{uniquenessdisk}
If   $u, v:\DDDD$ are  harmonic sections  with sub-logarithmic growth  with   $u=v$ on $\partial \D$, then $u=v$ on $\D^*$.
\end{lemma}
\begin{proof}
The function $d^2(u,v)$ (as defined by (\ref{defdis}))
 is a continuous subharmonic function (cf.~\cite[Remark 2.4.3]{korevaar-schoen1}).  Furthermore,  $d^2(u,v)\equiv 0$ on $\partial \D$.  Since $u$, $v$ both have sub-logarithmic growth,   the triangle inequality implies  
 \[
 \lim_{|z| \rightarrow 0} d^2(u(z),v(z)) +\epsilon \log |z| =-\infty.
 \]
Thus, we can apply Lemma~\ref{shdisk}  to conclude $d^2(u(z),v(z))\equiv 0$ on $ \D^*$ which implies $u \equiv v$.
\end{proof}

\begin{corollary} \label{converse}
Any harmonic section  $v:\DDDD$ 
with sub-logarithmic growth satisfies properties (i), (ii) and (iii)  of Theorem~\ref{exists}.
\end{corollary}

\begin{proof}
By the uniqueness assertion of Lemma~\ref{uniquenessdisk}, $v$ must be the harmonic section $u$ constructed in Theorem~\ref{exists}.
\end{proof}

 \section{Existence of infinite energy harmonic maps} \label{sec:proofexistence}

In this section, we prove existence of equivariant harmonic maps from the punctured Riemann surface  $\RR$. 
We  use the following notation:
\begin{itemize}
\item $\D^{j*}=\D^j \backslash \{0\}$
\item $\D^j_r=\{z \in \D^j: |z|<r\}$
\item $\D^j_{r,r_0}=\D^j_{r_0} \backslash \D^j_r$
\item $\RR_r=\RR \backslash \bigcup_{j=1}^n \D^j_r$
\item $2\pi \Z \simeq \langle \lambda^j  \rangle$ is  the free group generated by a loop around the puncture $p^j$
\item $\rho^j:\langle \lambda^j \rangle \rightarrow \mathsf{Isom}(\tilde X)$ be the restriction of $\rho: \pi_1(\RR) \rightarrow \mathsf{Isom}(\tilde X)$
\item $k: \RR \rightarrow  \tilde{\mathcal R} \times_\rho \tilde X$ is a locally Lipschitz section  (cf.~\cite[Proposition 2.6.1]{korevaar-schoen1})
\item   $k^j:=k|_{\bar \D^{j*}}$, restriction to the conformal disk around the puncture $p^j$

\end{itemize}
Applying   Lemma~\ref{defineg} with $\rho=\rho^j$, $I=I^j$ and $k=k^j$ yields  a prototype  section 
\[
v^j:  \D^* \rightarrow \widetilde{ \D^*} \times_{\rho^j} \tilde X \subset  \widetilde{{\mathcal R}} \times_{\rho^j} \tilde X.
\]
The composition of $v^j$ and the quotient map 
$\widetilde{{\mathcal R}} \times_{\rho^j} \tilde X \rightarrow  \widetilde{{\mathcal R}} \times_{\rho} \tilde X$ defines 
a section
 of $\tilde {\mathcal R} \times_\rho \tilde X \rightarrow \mathcal R$ over  ${\D}^{j*}$ which we call again $v^j$.
We extend  these local sections $v^j:\D^{j*} \rightarrow \tilde {\mathcal R} \times_{\rho} \tilde X$ for $j=1,\dots, n$  to define a smooth section $v:{\mathcal R} \rightarrow   \tilde {\mathcal R}\times_{\rho} \tilde X$.
By the construction and Lemma~\ref{defineg},
\begin{equation} \label{ebv}
\sum_{j=1}^n\Ej \log \frac{1}{r} \leq 
E^v[\RR_r] \leq
 \sum_{j=1}^n \Ej \log \frac{1}{r}+C, \ \ \ 0<r
\leq  1.
\end{equation}
where $\Ej:=\frac{\Delta_j^2}{2\pi}$.

\begin{definition}  
The locally Lipschitz section 
\begin{equation} \label{mochizukimap}
v:  \mathcal R \rightarrow \tilde{\mathcal R} \times_{\rho} \tilde{X}
\end{equation}
constructed above  is called the  {\it prototype section} of the fiber bundle $\tilde{\mathcal R} \times_{\rho} \tilde{X} \rightarrow \mathcal R$. 
The associated $\rho$-equivariant map $\tilde{v}:\tilde{\mathcal R} \rightarrow \tilde{X}$  is called the {\it prototype map}.   
\end{definition}

Define
\[
u_{\RR_r}:   \RR_r \rightarrow \tilde \RR \times_{\rho} \tilde X
\]
 to be the unique harmonic section with boundary values equal to that of $v|_{\partial \RR_r}$  (cf.~\cite[Theorem 2.7.2]{korevaar-schoen1}).  Fix $r_0 \in (0,\frac{1}{2}]$ and let $r \in (0,r_0)$.  
By Lemma~\ref{lowerbd},
\[
\Ej\log \frac{r_0}{r} 
  \leq    \int_{\D_{r,r_0}} \frac{1}{r^2} \left| \frac{\partial u_{\RR_r}}{\partial \theta} \right|^2 rdr \wedge d\theta\leq  E^{u_{\RR_r}}[\D^j_{r,r_0}].
\]
Thus, by (\ref{ebv}),
\[
E^v[\bigcup_{j=1}^n \D^j_{r,r_0}] \leq E^{u_{\RR_r}}[\bigcup_{j=1}^n\D^j_{r,r_0}]+C
\]
 which implies 
\begin{eqnarray}
E^{u_{\RR_r}}[\bigcup_{j=1}^n \D^j_{r,r_0}] +E^{u_{\RR_r}}[\RR_{r_0}]& = & E^{u_{\RR_r}}[\RR_r]\nonumber  \\
 & \leq  &  E^v[\RR_r] \nonumber  \\
 & = &  
E^v[\bigcup_{j=1}^n \D^j_{r,r_0}] +E^v[\RR_{r_0}]  \nonumber\\
& \leq &  
E^{u_{\RR_r}}[\bigcup_{j=1}^n \D^j_{r,r_0}] +C+E^v[\RR_{r_0}].   \nonumber  
 \end{eqnarray}
In other words,
 \begin{equation} \label{bdvr}
 E^{u_{\RR_r}}[\RR_{r_0}] \leq C+ E^v[\RR_{r_0}], \ \ \ \forall r \in (0,r_0).
 \end{equation}
The right hand side of the inequality (\ref{bdvr})  is independent of the parameter $r$; i.e.~once we fix $r_0$, the quantity $E^{u_{\RR_r}}[\RR_{r_0}]$ is uniformly bounded for all $r\in (0, r_0)$.  This implies a uniform Lipschitz bound, say $L$,  of $u_{\RR_r}$ for $r \in (0,r_0)$ in $\RR_{2r_0}$ (cf.~\cite[Theorem 2.4.6]{korevaar-schoen1}).  

Let $\{\mu_1, \dots, \mu_N\}$ be a set of generators of $\pi_1(\RR)$ and  $\tilde u_{\RR_r}$, $\tilde v$ be the $\rho$-equivariant maps associated to  sections $u_{\RR_r}$, $v$ respectively.
Thus,
  \[
d(\tilde u_{\RR_r}(\mu_i p), \tilde u_{\RR_r}(p)) \leq L d_{\tilde \RR}(\mu_i p,p),
\ \ 
\forall p \in \tilde \RR_{2r_0}, \ i=1, \dots, N \ \ r \in (0,r_0).
\]
If we let  
\[
c=L \sup \{d_{\tilde \RR}(\mu_i p,p): i=1,\dots, N, \ p \in \tilde \RR_{2r_0} \},
\]
then by  equivariance
\[
d(\rho(\mu_i)\tilde u_{\RR_r}(p),\tilde u_{\RR_r}(p)) \leq c, \ \ \ 
p \in \tilde \RR_{2r_0}, \ i=1, \dots, N, \ \ r \in (0,r_0).
\]
In other words, $\delta(\tilde u_{\RR_r}(p)) \leq c$ for all $p \in \tilde \RR_{2r_0}$ and $r \in (0,r_0)$.
By the properness of $\rho$, there exists $P_0 \in \tilde X$ and $R_0>0$ such that 
\[
\{\tilde u_{\RR_r}(p): \  p \in \tilde \RR_{2r_0}, \ r \in (0,r_0)\} \subset B_{R_0}(P_0).
\]
Thus,  by taking a compact exhaustion and applying \cite[Theorem 2.1.3]{korevaar-schoen2},  we conclude that there exists a sequence $r_i \rightarrow 0$ and a $\rho$-equivariant harmonic map $\tilde u: \tilde \RR \rightarrow
\tilde X$ such that $u_{\RR_{r_i}}$ converges to $u$ in $L^2$ (i.e.~$d^2(u_{\RR_{r_i}}, u) \rightarrow 0$) on every compact subsets of $\RR$.  Let $u: \RR \rightarrow \tilde \RR \times_{\rho} \tilde{X}$ be the associated  harmonic section.
The lower semicontinuity of energy (cf.~\cite[Theorem 1.6.1]{korevaar-schoen1}) and (\ref{bdvr}) imply
\[
E^{u}[\RR_{r_0}] \leq  \sum_{j=1}^n\Ej  \log \frac{1}{r_0}+C, 
\ \ \forall r_0 \in (0,\frac{1}{2}]
\]
which, along with Lemma~\ref{lowerbd}, proves that $u$ has  logarithmic energy growth (cf.~Definition~\ref{logdec}).  The fact that $u$ has sub-logarithmic growth follows from Lemma~\ref{(iii)}.  This completes the proof of  Theorem~\ref{existence}.

\section{Uniqueness of infinite energy harmonic maps} \label{sec:proofuniqueness}

In this section, we prove uniqueness of  harmonic maps from  punctured Riemann surfaces  $\RR$ (cf.~Theorem~\ref{uniqueness}).    First, let 
\[
E^f(r)=E^f[\RR_r] - \sum_{j=1}^n \Ej \log \frac{1}{r}.
\]
By Lemma~\ref{lowerbd},  for $ 0<r<r_0<1$ and any section $f:\RR \rightarrow \tilde \RR \times_\rho \tilde X$, 
\begin{equation} \label{lwrbd}
\sum_{j=1}^n\Ej \log \frac{r_0}{r}   \leq  E^f[\bigcup_{j=1}^n \D^j_{r,r_0}].
\end{equation}
Thus,  $r \mapsto E^f(r)$ is an increasing function.  
\begin{definition} \label{minuslogterm} 
The {\it modified energy} of $f$ is 
\[
E^f(0) = \lim_{r \rightarrow 0} E^f(r).
\]
Let $\LL_\rho$ be the set of all sections $\RR \rightarrow \tilde \RR \times_\rho \tilde X$ such that  $E^f(0)<\infty$.  Let $\mathcal H_{\rho} \subset \mathcal L_{\rho}$ denote the set of all   harmonic sections $u:  \RR \rightarrow \tilde \RR \times_\rho \tilde X$ of sub-logarithmic growth such that the associated $\rho$-equivariant map $u$ that is not identically constant or does not map into a geodesic. 
\end{definition}

  The goal is to prove that $\mathcal H_{\rho}$ contains only one element.   We first prove the following series of preliminary lemmas.

\begin{lemma} \label{u1}  If $u \in {\mathcal H}_{\rho}$ and $\D^j \subset \bar \RR$ is  the fixed conformal disk at the puncture $p^j$ (cf.~(\ref{fixeddisk})), then the restriction map $u|_{\D^{j*}}$ satisfies the properties (i), (ii) and (iii)  of Theorem~\ref{exists}. 
\end{lemma}

\begin{proof}
This follows immediately from Corollary~\ref{converse}.
\end{proof}

\begin{lemma} \label{sh01}  
If $u_0, u_1 \in {\mathcal H}_{\rho}$,  then
$
d^2(u_0,u_1)=c
$
for some constant $c$.  
\end{lemma}

\begin{proof}
We prove  Lemma~\ref{sh01} by showing that $d^2(u_0,u_1)$ extends as a subharmonic function to $\bar \RR$.  Since $\bar \RR$ is compact,  this implies  $d^2(u_0,u_1)$ is constant.  Indeed,   let   $\D \subset \RR$ be a holomorphic disk.  By  \cite[Lemma 2.4.2 and Remark 2.4.3]{korevaar-schoen1},    the function $d^2(u_0,u_1)$ is subharmonic in $\D$.  Since this statement is true for any holomorphic disk $\D \subset \RR$, we conclude that $d^2(u_0,u_1)$ is subharmonic in $\RR$.  

Next, consider the fixed conformal disk $\D^j \subset \bar{\RR}$  centered at   $p^j \in {\mathcal P}$ (cf.~(\ref{fixeddisk})).    
Since both $u_0$ and $u_1$ have sub-logarithmic growth, 
$d(u_0(z),u_1(z))+\epsilon \log  |z| \rightarrow -\infty$ as $z \rightarrow 0$ in $\D^j$ by the triangle inequality.  Thus,  by Lemma~\ref{shdisk},
$d^2(u_0,u_1)$   is bounded in $\D^{j*}$ 
and  extends as a subharmonic function on $\D^j$.  Hence, we conclude that $d^2(u_0,u_1)$ extends as a subharmonic function on $\bar \RR$, thereby proving    Lemma~\ref{sh01}.
\end{proof}

For $u_0, u_1 \in {\mathcal H}_{\rho}$, let $\tilde u_0$, $\tilde u_1$ be the associated $\rho$-equivariant maps.   For $s \in [0,1]$, define  $u_s:\RR \rightarrow  \tilde \RR \times_{\rho} \tilde X$  to be the associated section of the $\rho$-equivariant map
  \begin{equation} \label{us}
 \tilde u_s:  \tilde \RR \rightarrow \tilde X, \ \ \ \tilde u_s(z)=(1-s)\tilde u_0(z)+s \tilde u_1(z)
  \end{equation}
where the sum on the right hand side above denotes geodesic  interpolation (cf.~Notation~\ref{interpolationnotation}). 
Since $\tilde u_0$ and $\tilde u_1$ are $\rho$-equivariant, $\tilde u_s$ is also $\rho$-equivariant.  
 
 From Lemma~\ref{u1}, Lemma~\ref{sh01}, the convexity of the distance function and the convexity of  energy (cf.~\cite[(2.2vi)]{korevaar-schoen1}), it follows that $u_s|_{\D^{j*}}$ also satisfies the properties $(i)$, $(ii)$ and $(iii)$ of Theorem~\ref{exists} for all $s \in [0,1]$; i.e.
\begin{itemize}
\item
$\displaystyle{
E^{u_s}[\D^j_{r,r_0}] \leq  C+ \Ej \log \frac{r_0}{r}
}
$
\vspace{0.5mm}
\item 
$\displaystyle{
\lim_{r \rightarrow 0} E^{u_s|_{\partial \D^j_r}}[\Sp^1]  =  \Ej,
}
\vspace{0.5mm}
$
\item $u_s$ has sub-logarithmic growth.
 \end{itemize}
 
\begin{lemma}  \label{claim:dist}
Let $u_0, u_1 \in {\mathcal H}_{\rho}$ and $\epsilon>0$.  For any $\rho_0 \in (0,1)$, there exists $r_0 \in (0,\rho_0)$ such that
\[
\frac{1}{-\log r_0}  \sum_{j=1}^n \int_0^{2\pi} d^2(u_0|_{\partial \D^j_{r_0}}, u_{s}|_{\partial \D^j_{r_0^2}})d\theta<\epsilon,
\ \  \ \forall s \in [0,1].
\]
\end{lemma}

\begin{proof}
It suffices to prove Lemma~\ref{claim:dist} for $s=1$. Let $\rho_0>0$ be given. First, Lemma~\ref{des} asserts that there exists $\rho_1 \in (0,\rho_0)$ such that 
\[
\frac{1}{-\log r_0}  \sum_{j=1}^n \int_0^{2\pi} d^2(u_1|_{\partial \D^j_{r_0}}, u_1|_{\partial \D^j_{r_0^2}})d\theta
< \frac{\epsilon}{4}, \ \  \ \ \forall r_0 \in (0,\rho_1). 
\]
For $c>0$ as in Lemma~\ref{sh01}, choose $r_0 \in (0,\rho_1)$ such that 
$
 \frac{2\pi  n c}{-\log r_0}<\frac{\epsilon}{4}.
 $
Then 
\[
\frac{1}{-\log r_0} \sum_{j=1}^n \int_0^{2\pi} d^2(u_0|_{\partial \D^j_{r_0}}, u_1|_{\partial \D^j_{r_0}})d\theta
=\frac{2\pi n c}{-\log r_0} < \frac{\epsilon}{4}.
\]
The inequality of Lemma~\ref{claim:dist} for $s=1$ follows from the above two inequalities and the triangle inequality.
\end{proof}

\begin{lemma} \label{claim:lim}
Let $u_0, u_1 \in {\mathcal H}_{\rho}$ and $s \in [0,1]$.  For $\epsilon>0$, there exists  $\rho_0>0$ sufficiently small such that 
\[
-\log r_0  \sum_{j=1}^n \left(E^{u_0|_{\partial \D^j_{r_0}}}[\Sp^1] +E^{u_s |_{\partial \D^j_{r_0^2}}}[\Sp^1] \right) <-2 \log r_0 \sum_{j=1}^n \Ej +\epsilon, \ \ \ 0<r_0<\rho_0.
\]
\end{lemma}

\begin{proof}
For $s=1$,  Lemma~\ref{claim:lim}  follows from Lemma~\ref{(ii)}.  The general case of $s \in [0,1]$ follows immediately by convexity of energy (cf.~\cite[(2.2vi)]{korevaar-schoen1}). 
\end{proof} 

\begin{lemma} \label{sae}   For $u_0, u_1 \in {\mathcal H}_{\rho}$ and $s \in [0,1]$,  we have
$
E^{u_0}(0) = E^{u_s}(0).
$
\end{lemma}

\begin{proof}
It suffices to prove Lemma~\ref{sae} for $s=1$.   Assume on the contrary  that $E^{u_0}(0) \neq  E^{u_1}(0)$.  By changing the role of $u_0$ and $u_1$ if necessary, we can assume $E^{u_0}(0) < E^{u_1}(0)$.  Let  $\rho_0>0$ smaller of  the $\rho_0$ in Lemma~\ref{claim:dist} and Lemma~\ref{claim:lim}.  Choose $\epsilon>0$  and $r_0 \in (0,\rho_0]$ such that
\[
E^{u_0}(r_0) <  E^{u_1}(r_0) -2\epsilon
\]
which implies (cf.~Definition~\ref{minuslogterm})\begin{equation} \label{ococ}
E^{u_0}[\RR_{r_0}] <  E^{u_1}[\RR_{r_0}] -2\epsilon.
\end{equation}
Next fix $r_1 \in (0,r_0)$.  Let  $\tilde u_0$, $\tilde u_1:  \tilde \RR \rightarrow \tilde X$ be the $\rho$-equivariant maps associated to sections $u_0$, $u_1$. 
 For the fixed conformal disk $\D^j \subset \bar \RR$ centered at the puncture $p^j$ (cf.~(\ref{fixeddisk})), let  $ \tilde \D^j_{r_1,r_0} \subset \tilde \RR$ the lift of   $\D^j_{r_1,r_0} \subset \RR$.
We define a ``bridge" between  map $\tilde u_0$ and $\tilde u_1$  by setting 
\[
\tilde b: \bigcup_{j=1}^n \tilde \D^j_{r_1,r_0}  \rightarrow \tilde X
\] 
 to be the geodesic  interpolation $(1-t) U_0(\theta) +tU_1(\theta)$ where
  \[
 t=\frac{\log |z|-\log r_0}{\log r_1 - \log r_0}, \ \   U_0(\theta) = \tilde u_0(r_0,\theta) \ \text{ and } \  U_1(\theta) = \tilde u_1(r_1,\theta).
 \]
In other words,
\[
\tilde b(r,\theta)=\frac{\log r_1 - \log r}{\log r_1 - \log r_0} \tilde u_0(r_0, \theta)+\frac{\log r - \log r_0}{\log r_1 - \log r_0} \tilde u_1(r_1, \theta)
\]
 for $(r,\theta)  \in \tilde \D^j_{r_1,r_0}$,  $j=1,\dots, n$.  
Let
 \[
b:\bigcup_{j=1}^n \D^j_{r_1,r_0} \rightarrow 
\tilde \RR \times_{\rho}  \tilde X
\]
be the local section associated with $\tilde b$.

The CAT(0) condition implies (by an argument analogous to the proof the  bridge lemma  \cite[Lemma 3.12]{korevaar-schoen2})
\begin{eqnarray*}
E^b[\bigcup_{j=1}^n \D^j_{r_1,r_0}] 
& \leq & 
\frac{1}{2} \log \frac{r_0}{r_1} \sum_{j=1}^n \left(E^{u_0|_{\partial \D^j_{r_0}}}[\Sp^1] +E^{u_1|_{\partial \D^j_{r_1}}}[\Sp^1] \right) \\
& & \  + \frac{1}{\log \frac{r_0}{r_1}} \sum_{j=1}^n \int_0^{2\pi} d^2(u_0|_{\partial \D^j_{r_0}}, u_1|_{\partial \D^j_{r_1}})d\theta.
\end{eqnarray*} 

Choose $r_1=r_0^2$ (cf.~ Lemma~\ref{claim:dist} and Lemma~\ref{claim:lim}) to obtain 
 \begin{equation} \label{acac}
E^b[\bigcup_{j=1}^n \D^j_{r_0^2,r_0}] < -\log r_0
\sum_{j=1}^n \Ej  +2\epsilon.
\end{equation}
Since the section 
\[
h:\RR_{r_0^2} \rightarrow \tilde \RR \times_\rho \tilde X
\]
 defined by setting 
 \[
 h=
 \left\{ 
 \begin{array}{ll}
 u_0 & \mbox{ in }\RR_{r_0}\\
b & \mbox{ in } \displaystyle{\bigcup_{j=1}^n \D^j_{r_0^2,r_0}}
\end{array}
\right.
\] is a competitor for $u_1$, we have
\begin{eqnarray*}
E^{u_1}[\RR_{r_0^2}]  & \leq &  E^h[\RR_{r_0^2}]
\\
& = & E^{u_0}[\RR_{r_0}]+ E^b[\bigcup_{j=1}^n \D^j_{r_0^2,r_0}]\\
&  <&  E^{u_1}[\RR_{r_0}]-2\epsilon + E^b[\bigcup_{j=1}^n \D^j_{r_0^2,r_0}] \ \ \ \ \ \ \mbox{(by (\ref{ococ}))}\\
& < & E^{u_1}[\RR_{r_0}] -\log r_0 \sum_{j=1}^n\frac{\Delta_{I^j}^2}{2\pi}  \ \ \ \ \ \ \mbox{(by (\ref{acac}))}.
\end{eqnarray*}
Thus,
\[
E^{u_1}[\bigcup_{j=1}^n \D^j_{r_0^2,r_0}] < -\log r_0 \sum_{j=1}^n \Ej.
\]
This contradicts   (\ref{lwrbd}) and proves Lemma~\ref{sae}.
\end{proof}

\begin{lemma} \label{twothings}
For $u_0, u_1 \in \HH_\rho$, there exists a constant $c$ such that
\begin{eqnarray} 
d(u_s(p), u_1(p)) & \equiv  & cs, \forall p \in  \RR
\label{d=1}
 \\
 |(u_s)_*(V)|^2(p) & = & |(u_0)_*(V)|^2(p), \ \ \mbox{for a.e.~}s \in [0,1], \ p \in \RR, V \in T_p\tilde \RR. \label{pullback}
\end{eqnarray}
\end{lemma}

\begin{proof}
By the  convexity of energy (cf.~\cite[(2.2vi)]{korevaar-schoen1}), 
\[
E^{u_s}[\RR_r] \leq (1-s)E^{u_0}[\RR_r] + s E^{u_1}[\RR_r] -s(1-s) \int_{\RR_r} |\nabla d(u_0, u_1)|^2 d\mbox{vol}_{\domain}
\]
for any $r>0$.
\[
E^{u_s}(r) \leq (1-s)E^{u_0}(r) + s E^{u_1}(r) -s(1-s) \int_{\RR_r} |\nabla d(u_0, u_1)|^2 d\mbox{vol}_{\domain}.
\]
Letting $r \rightarrow 0$ and applying Lemma~\ref{sae},  we conclude 
\begin{eqnarray}
0 &= & \int_{\RR} |\nabla d(u_0, u_1)|^2 d\mbox{vol}_{\domain}
 \label{graddist}
 \\
E^{u_s}& = & E^{u_0}, \ \ \forall s \in [0,1] \label{energyconstant}
 \end{eqnarray}

First, (\ref{graddist}) implies that $\nabla d(u_0, u_1)=0$~a.e.~in $\domain$ which in turn implies $d(u_0,u_1) \equiv c$ for some $c$.  By the definition of the map $u_s$, (\ref{d=1}) follows immediately.

Next, for $\{P,Q,R,S\} \subset \tilde X$, the quadrilateral comparison for CAT(0) spaces  (cf.~\cite{reshetnyak})  implies
\[
d^2(P_t,Q_t) \leq (1-t) d^2(P,Q)+td^2(R,S)
\]
where $P_t=(1-t)P+tS$ and $Q_t=(1-t)Q+tR$.
Applying the above inequality with  $P=u_0(p)$, $Q=u_1(p)$, $S=u_1(\exp_p(tV))$ and $Q=u_0(\exp_t(tV))$ where $t >0$ and $V \in T_p\tilde M$, dividing by $y$ and letting $t \rightarrow 0$, we obtain (cf.~\cite[Theorem 1.9.6]{korevaar-schoen1})
\[
|(u_s)_*(V)|^2(p) \leq (1-s) |(u_0)_*(V)|^2(p) + s|(u_1)_*(V)|^2(p),
\ 
\mbox{a.e.~$p \in \tilde M, V \in T_p\tilde M$.}
\]   Integrating the above over all unit vectors $V \in T_p\tilde M$  and then over $p \in F$, we obtain
\[
E^{u_s} \leq (1-s) E^{u_0}+ s E^{u_1}.
\]
Combining this with  (\ref{energyconstant}) implies (\ref{pullback}).
\end{proof}

\begin{proof1.2}
 We assume there exist $u_0, u_1 \in \HH_\rho$ such that $u_0 \not \equiv u_1$ and treat the different cases of Theorem~\ref{uniqueness} separately. 
 
 \begin{itemize}
 \item {\it {\bf $\tilde X$ is a negatively curved space}}:  
For  $p \in \tilde \RR$,  choose $\kappa>0$ and $R>0$ such that $B_{R}(u_0(p))$ is a CAT(-$\kappa$) space.  Let   $\mathcal U$ be a sufficiently small neighborhood of  $p$  and $s_0 \in [0,1]$ sufficiently small such that $\tilde u_s(\mathcal U) \subset B_{R}(u(p))$ for $s \in [0,s_0]$.  Thus, if $\tilde u_{s_0}(p)\neq \tilde u_0(p)$, then applying  \cite[Section 5]{mese} implies that the image under $u_0$ of a  sufficiently small  neighborhood of $p$ is contained in an image $\sigma(\R)$ of a geodesic line.   (If $\tilde X$ is a smooth manifold, this follows by Hartman \cite{hartman}).  If $ \tilde u_0 \not \equiv \tilde u_1$, then $\tilde u_s(p) \neq \tilde u_0(p)$ for all $p \in \tilde \RR$ and $s \in (0,s_0]$ by (\ref{d=1}).  Thus, we conclude that $\tilde u_0(\tilde \RR)$ is contained in $\sigma(\R)$.  Consequently, $\rho(\pi_1(\tilde \RR))$ fixes $\sigma(\R)$ which contradicts  the fact that $\rho(\pi_1(\RR))$ is satisfies assumption (B). 

\item {\bf $\tilde X$ is an irreducible symmetric space of non-compact type or locally finite Euclidean buildings}:
For these two target spaces, the conclusion of  Lemma~\ref{twothings}  is the same as that of \cite[Lemma 3.1]{daskal-meseUnique}.  Thus, we can apply the arguments of  \cite[Section 3.2 and Section 3.3]{daskal-meseUnique}.
\end{itemize}

\end{proof1.2}

\end{document}